\documentclass[12pt]{amsart}

\usepackage{fullpage}
\usepackage{mathtools}
\usepackage{amssymb,amsmath,amsthm,amscd,mathrsfs,graphicx}
\usepackage[dvipsnames]{xcolor}
\usepackage{bm}
\usepackage{dsfont}
\usepackage{enumerate}
\usepackage{epsfig}
\usepackage{float, graphicx}
\usepackage{latexsym, amsxtra}
\usepackage{pdfpages}
\usepackage{multicol}
\usepackage[normalem]{ulem}
\usepackage{psfrag}
\usepackage{stmaryrd}
\usepackage{tikz}
\usepackage[T1]{fontenc}
\usepackage{url}
\usepackage{verbatim}
\usepackage{xy}
\usepackage{indentfirst}
\usepackage{tikz-cd}
\usepackage{url}

\usepackage[
  colorlinks=true,
  linkcolor=red,
  urlcolor=blue,
  citecolor=blue
]{hyperref}

\makeatletter
\def\thm@space@setup{%
  \thm@preskip=2ex \thm@postskip=2ex
}
\makeatother

\numberwithin{equation}{section}
\theoremstyle{plain}

\newtheorem{thm}{Theorem~}[section] 
\newtheorem{lem}[thm]{Lemma~}

\newtheorem{prop}[thm]{Proposition~}

\theoremstyle{remark}
\newtheorem{rmk}[thm]{Remark~}

\theoremstyle{definition}
\newtheorem{defn}[thm]{Definition~}

\newcommand{\calD}{\mathcal{D}}

\newcommand{\calC}{\mathcal{C}}

\newcommand{\calL}{\mathcal{L}}

\newcommand{\CC}{\mathbb{C}}
\newcommand{\ZZ}{\mathbb{Z}}
\newcommand{\RR}{\mathbb{R}}
\newcommand{\LL}{\mathbb{L}}
\newcommand{\PP}{\mathbb{P}}

\newcommand\Co{\mathrm{Co}}

\newcommand\rank{\mathrm{rank}}

\newcommand\II{\mathrm{II}}

\newcommand\rO{\mathrm{O}}

\DeclareMathOperator{\Aut}{Aut}

\title{A Lemma on Leech-like Lattices}

 \author[Z. Zheng]{Zhiwei Zheng}
\address{Tsinghua University, China}
\email{zhengzhiwei@mail.tsinghua.edu.cn}
\date{}

\begin{document}
\bibliographystyle{amsalpha}

\begin{abstract} 
A Leech pair is defined as a pair $(G,S)$, where $S$ is a positive definite even lattice without roots, equipped with a faithful action of a finite group $G$, such that the invariant sublattice of $S$ under the action of $G$ is trivial, and the induced action of $G$ on the discriminant group of $S$ is also trivial. This structure appears naturally when investigating hyperk\"ahler manifolds and the symplectic automorphisms acting on them. An important lemma due to Gaberdiel--Hohenegger--Volpato asserts that a Leech pair $(G,S)$ admits a primitive embedding into the Leech lattice if $\rank(S)+\ell(A_S)\le 24$. However, the original proof is incomplete, as demonstrated by a counterexample provided by Marquand and Muller. They also presented a computer-assisted proof of the lemma for cases where $\rank(S) \leq 21$. In this paper, we modify the original approach to provide a complete and conceptual proof of the lemma.
\end{abstract}

\maketitle

%\tableofcontents
\section{Introduction}
In this paper we provide a conceptual proof for the following result:
\begin{lem}
\label{lemma: main}
Suppose $(G,S)$ is a Leech pair (see Definition \ref{definition: leech pair}) such that $\rank(S)+\ell(A_S)\le 24$, then there exists a primitive embedding of $S$ into the Leech lattice $\LL$. Under this embedding, the action of $G$ on $S$ extends to an action on $\LL$ such that $G$ acts trivially on the orthogonal complement of $S$ in $\LL$.
\end{lem}

Lemma \ref{lemma: main} was first stated by Gaberdiel, Hohenegger and Volpato \cite[Appendix B.2]{gaberdiel2012symmetries}. Recently, Marquand and Muller observed that the original proof is incomplete. Specifically, the argument in \cite{gaberdiel2012symmetries} only establishes the existence of an embedding of $S$ into the Leech lattice, but does not show the primitiveness of such an embedding. As noted in \cite[Example 4.12]{marquand2025finite}, the approach of \cite{gaberdiel2012symmetries} does not directly yield the desired primitiveness condition.

In \cite[Corollary 4.19]{marquand2025finite}, Marquand and Muller established Lemma~\ref{lemma: main} for $\rank(S) \leq 21$ using a computer-assisted approach. This result is crucial to the present work, as it provided strong evidence that a conceptual proof of Lemma~\ref{lemma: main} should exist.

Lemma \ref{lemma: main} has been used by many subsequent papers; see \S\ref{section: leech pair} for a discussion. These applications only rely on the cases with $\rank(S)\le 21$. Thanks to \cite[Corollary 4.19]{marquand2025finite}, these results remain valid. 

%This paper provides a conceptual proof for Lemma \ref{lemma: main}.
%The gap in the original proof in \cite[Appendix B, Proposition 1]{gaberdiel2012symmetries} is as follows. When one finds a Weyl vector $w\in \II_{25,1}$ orthogonal to $S$, the composition of $S\hookrightarrow w^{\perp}\to w^{\perp}/\ZZ w$ is an embedding of $S$ into the Leech lattice $\LL$, but not necessarilly primitive.

In this paper, we modify the original proof in \cite{gaberdiel2012symmetries} and provide a conceptual proof for Lemma \ref{lemma: main}. This establishes an alternative proof of Lemma \ref{lemma: main} for cases $\rank(S)\le 21$, and completes the proof for the remaining cases $22\le \rank(S)\le 24$. These latter cases may become relevant in future work on hyperk\"ahler varieties in positive characteristic.

We will prove the following extended version of Lemma \ref{lemma: main}:

\begin{prop}
\label{proposition: extended version}
Suppose $(G,S)$ is a Leech pair, then the following statements are equivalent:
\begin{enumerate}
\item $\rank(S)+\ell(A_S)\le 24$.
\item There exists a primitive embedding $S\oplus A_1\hookrightarrow \II_{25,1}$.
\item There exists a primtive embedding of $S$ into the Leech lattice.
\end{enumerate}
\end{prop}

We explain the idea for the proof of Proposition \ref{proposition: extended version}. The implication $(3)\Rightarrow (1)$ is straightforward. The implication $(1)\Rightarrow (2)$ was proved by Gaberdiel--Hohenegger--Volpato \cite[Proof of Proposition B.2] {gaberdiel2012symmetries} based on Nikulin criterion \cite[Theorem 1.12.2]{nikulin1980integral}. So the main part we will deal with is the implication $(2)\Rightarrow (3)$. In \cite[Appendix B.2] {gaberdiel2012symmetries}, Gaberdiel, Hohenegger and Volpato actually proved that: 

\begin{prop}[\cite{gaberdiel2012symmetries}]
\label{proposition: gaberdiel}
For a Leech pair $(G,S)$, if there exists a primitive embedding of $S$ into $\II_{25,1}$, then there exists an embedding (which is not necessarily primitive) of $S$ into the Leech lattice $\LL$.
\end{prop}

As demonstrated by the counterexample \cite[Example 4.12]{marquand2025finite}, conditions in Proposition \ref{proposition: gaberdiel} cannot ensure existence of a primitive embedding of $S$ into $\LL$. What we will argue in \S\ref{section: proof} is to fully use the statement $(2)$ in Proposition \ref{proposition: extended version} to construct a primitive embdding of $S$ into the Leech lattice. The key idea is a careful selection of a Conway chamber with respect to which a generator of $A_1$ maps to a simple root in $\II_{25,1}$.

\begin{rmk}
This remark was kindly communicated to me by Stevell Muller. All the statements (Lemma~\ref{lemma: main}, Proposition~\ref{proposition: extended version}, and Proposition~\ref{proposition: gaberdiel}) are really about the lattice $S$. The condition that $S$ fits into a Leech pair $(G,S)$ (see Definition~\ref{definition: leech pair}) is equivalent to the requirement that the action of $\widetilde{\mathrm{O}}(S)$ (the group of automorphisms of $S$ that act trivially on the discriminant group $A_S$) on $S$ has no nonzero fixed vectors. This equivalence allows us to formulate the results without referring to the group $G$. However, we have kept the notion of a Leech pair to stay consistent with the existing literature.
\end{rmk}

\section{Leech Pair}
\label{section: leech pair}
We now introduce some terminology that will be used. We refer the reader to \cite{nikulin1980integral} for more details on integral lattices. For an even lattice $M$, we denote by $A_M$ its discriminant group. Let $\rO(M)$ denote the group of isometries of $M$, and let $\widetilde{\rO}(M)$ denote the subgroup consisting of those isometries that act trivially on $A_M$. A positive definite even lattice is said to be rootless if it contains no vectors of norm two. 

There exists a unique (up to isomorphism) positive definite, rootless, even unimodular lattice of rank $24$, known as the Leech lattice, denoted by $\LL$. We set $\Co_0=\rO(\LL)$ and $\Co_1\coloneqq \Co_0/\{\pm 1\}$; both groups are referred to as Conway groups. In particular, $\Co_1$ is a sporadic simple group of order $2^{21}\cdot 3^9\cdot 5^4\cdot7^2 \cdot 11 \cdot 13 \cdot 23$.

For an integral lattice $M$ equipped with an action of a group $G$, we denote by $M^G$ the invariant sublattice, and let $S_G(M)$ be its orthogonal complement in $M$, referred to as the coinvariant sublattice. Both $M^G$ and $S_G(M)$ are primitive sublattices of $M$.

We now give the precise definition of a Leech pair, which was first introduced by Mongardi \cite[Definition 2.5.1]{mongardi2013automorphisms}.
\begin{defn} 
\label{definition: leech pair}
We call $(G,S)$ a Leech pair, if $S$ is a positive definite even rootless lattice, and $G$ is a subgroup of $\widetilde{\rO}(S)$ such that $S^G=(0)$.
\end{defn}

Then $(\Co_0, \LL)$ is a Leech pair. For a subgroup $G<\Co_0$, the pair $(G,S_G(\LL))$ is also a Leech pair, called a subpair of $(\Co_0, \LL)$. We call a subpair $(G, S_G(\LL))$ to be saturated, if $G=\widetilde{\rO}(S_G(\LL))$. H\"ohn and Mason \cite[Table 1]{HM16} classified all saturated subpairs of $(\Co_0, \LL)$, and found exactly $290$ cases.

Leech pairs arise naturally in the study of the classification of finite groups acting on various types of hyperk\"ahler geometries. They arise, for instance, in the classification of:
\begin{enumerate}
    \item finite groups of symplectic automorphisms of hyperk\"ahler manifold of type $K3^{[n]}$ \cite{huybrechts2016derived}\cite{mongardi2016towards}\cite{hohn2019finite},
    \item symplectic automorphism groups of smooth cubic fourfolds \cite{laza2022symplectic},
    \item finite groups of symplectic birational transforms of hyperk\"ahler manifold of $OG10$ type \cite{marquand2025finite},
    \item finite groups of symplectic biratonal automorphisms of hyperk\"ahler manifolds of type $K3^{[3]}$ \cite{billi2025EPW},
    \item finite groups of symplectic automorphisms of K3 surfaces in positive characteristic \cite{OS24}\cite{wang2024finite}.
\end{enumerate}
These works all rely on Lemma~\ref{lemma: main} for $\rank(S) \leq 21$, which was established by Marquand and Muller~\cite[Corollary~4.19]{marquand2025finite} using a computer-assisted method. By combining this with H\"ohn--Mason classification \cite[Table 1]{HM16}, one can analyze, and in many cases even classify, such finite groups using lattice-theoretic methods.

We now review the case of cubic fourfolds in more detail. In \cite{laza2022symplectic} the author and Laza classified the symplectic automorphism groups of all smooth cubic fourfolds $X$ in $\CC\PP^5$. An automorphsm of $X$ is called symplectic if it acts trivially on $H^{3,1}(X)$. Let $G=\Aut^s(X)$ denote the group of symplectic automorphisms of X. Then the pair $(G, S_G(H^4(X, \ZZ)))$ forms a Leech pair. By Lemma \ref{lemma: main}, the lattice $S_G(H^4(X, \ZZ))$ admits a primitive embedding into the Leech lattice. Therefore, the pair $(G, S_G(H^4(X, \ZZ)))$ is isomorphic to a subpair of $(\Co_0, \LL)$, and hence $G$ is isomorphic to a subgroup of the Conway group $\Co_0$. Moreover, this subpair can be shown to be saturated. Together with the H\"ohn--Mason classification \cite[Table 1]{HM16}, this implies that there are only finitely many possible isomorphism types for the pair $(G, S_G(H^4(X, \ZZ)))$. This argument plays a key role in the classification obtained in \cite{laza2022symplectic}.

\section{Borcherds Lattice and Conway Chamber}
\subsection{Weyl Chamber for Hyperbolic Lattices}
We first review the theory of Weyl groups and Weyl chambers for hyperbolic lattices. We refer the readers to \cite[\S1, Proposition 1.10]{Ogu83} and  \cite[Chapter 8, \S2]{Huy16} for more details. 

Let $L$ be an even lattice of signature $(n,1)$ with $n\ge 1$. Let $\calD$ and $-\calD$ be the two connected components of 
\[
\{x\in L_\RR\big{|} (x,x)<0\},
\]
and denote by $\overline{\calD}$ the closure of $\calD$ in $L_\RR$. We refer to $\calD$ as the Lobachevsky cone. 

An element $v\in L$ is called a root if $(v,v)=2$. Each root $v\in L$ defines a reflection hyperplane $H_v$, consisting of all points in $L_\RR$ that are orthogonal to $v$. The collection of all such  hyperplanes intersect with $\calD$ in a locally finite arrangement. A Weyl chamber for $\II_{25,1}$ is, by definition, a connected component of 
\[
\calD\cup (-\calD)-\bigcup_{v: \text{ roots in } L} H_v.
\]
The Weyl group $\mathcal{W}_L$ of $L$ is by definition the group generated by all reflections in roots. This group acts simply transitively on the set of Weyl chambers that are contained in $\calD$. 

Given a Weyl chamber $\calC$, we denote by $\overline{\calC}$ the closure of $\calC$ in $\calD\cup (-\calD)$. A root $v$ is called a simple root with respect to $\calC$ if the following two conditions are satisfied:
\begin{enumerate}[(i)]
\item the intersection $H_v\cap\overline{\calC}$ has codimension one in $L_\RR$, and
\item for all $x\in \overline{\calC}$, we have $(v,x)\le 0$.
\end{enumerate}
In this case we also say that $H_v$ is a wall supporting $\calC$.

\subsection{Borcherds Lattice and Conway Chamber}
There exists a unique even unimodular lattice $\II_{25,1}$ of signature $(25,1)$, known as the Borcherds lattice. In what follows, we take $L=\II_{25,1}$. For this particular choice of lattice, the corresponding Weyl chambers are referred to as Conway chambers.

A Weyl vector for the lattice $\II_{25,1}$ is a primitive isotropic vector $w$ such that the quotient lattice $w^{\perp}/\ZZ w$ is isomorphic to the Leech lattice. Any Weyl vector can be extended to a hyperbolic plane $U$ in $\II_{25,1}$, such that the orthogonal complement of $U$ in $\II_{25,1}$ is isomorphic to the Leech lattice. Therefore, all Weyl vectors for $\II_{25,1}$ form one $\mathrm{O}(\II_{25,1})$-orbit.  

Given a Weyl vector $w$, a root $r\in \II_{25,1}$ is called a Leech root with respect to $w$ if $(r,w)=-1$. Let $\calL_w$ denote the set of all Leech roots with respect to $w$. 

For a Weyl vector $w$, we put
\[
\calC_w\coloneqq \{x\in \calD\cup (-\calD) \big{|} (x, v)<0,\text{ for all } v\in \calL_w\}
\]

The following theorem due to Conway is crucial for the proof of Proposition \ref{proposition: gaberdiel}, and will also play an important role in the proof of Lemma \ref{lemma: main}.
\begin{thm}[Conway]
\label{theorem: conway}
The map $\varphi\colon w\mapsto \calC_w$ defines a bijection between the set of Weyl vectors and the set of Conway chambers. A Weyl vector $w$ belong to $\overline{\calD}$ if and only if $\calC_w$ is contained in $\calD$. Furthermore, a root $v$ is simple with respect to $\calC_w$ if and only if $v\in \calL_w$.
\end{thm}

\begin{proof}
This result is due to Conway \cite[Chapter 27, Theorem 1]{conway1999spherepackings}, although our formulation differs. For the reader's convenience, we provide an explanation of how \cite[Chapter 27, Theorem 1]{conway1999spherepackings} implies Theorem~\ref{theorem: conway}.

Let $\mathbb{R}^{25,1}$ denote the standard Lorentzian space with signature $(25,1)$. The even unimodular lattice $\II_{25,1}$ can be realized as a subset of $\mathbb{R}^{25,1}$ consisting of all vectors $(x_0, x_1, \cdots, x_{24} \mid x_{25})$ such that:
\begin{itemize}
\item each coordinate $x_i$ lies in $\ZZ$, or each lies in $\mathbb{Z} + \frac{1}{2}$, and
  \item the sum $x_0 + x_1 + \cdots + x_{24} + x_{25}$ is even, i.e., $\sum\limits_{i=0}^{25} x_i \in 2\ZZ$.
\end{itemize}
Then $w_0\coloneqq (0,1,\cdots, 24\mid 25)$ is a Weyl vector for $\II_{25,1}$. For this Weyl vector, Conway \cite[Chapter 27, Theorem 1]{conway1999spherepackings} proved that $\calC_{w_0}$ is a Conway chamber and $\calL_{w_0}$ is the set of simple roots with respect to $\calC_{w_0}$.

Recall that $\mathrm{O}(\II_{25,1})$ acts transitively on the set of all Weyl vectors of $\II_{25,1}$. Then for any Weyl vector $w$ in $\II_{25,1}$, there exists $g\in \mathrm{O}(\II_{25,1})$ such that $w=g(w_0)$. This implies that $\calL_{w}=g(\calL_{w_0})$ and $\calC_{w}=g(\calC_{w_0})$. Hence $\calC_w$ is a Conway chamber with $\calL_w$ the set of simple roots. Since $\mathrm{O}(\II_{25,1})$ acts transitively on the set of Conway chambers, we have surjectivity of $\varphi$.

Next we prove the injectivity of $\varphi$. Suppose two Weyl vector $w$ and $w'$ satisfy $\calC_w=\calC_{w'}$, then we have $\calL_w=\calL_{w'}$. Thus $w-w'$ is orthogonal to all elements of $\calL_w$. The set $\calL_w$, which consists of all simple roots with respect to $\calC_w$, spans the vector space $(\II_{25,1})_\RR$. Therefore, we must have $w=w'$. 

See also \cite[Remark 3.8]{brandhorst2023borcherds} for an alternative argument for the injectivity.
\end{proof}

\section{Proof of Lemma \ref{lemma: main} and Proposition \ref{proposition: extended version}}
\label{section: proof}

As explained before, we only need to prove $(2)\Rightarrow (3)$ in Proposition \ref{proposition: extended version}. Now we have a primitive embedding $S\oplus A_1\hookrightarrow \II_{25,1}$. Under this embedding, we abuse the notation and regard $S$ as a primitive sublattice of $\II_{25,1}$. Let $\alpha$ denote the image of one generator of $A_1$ under the embedding.

%We may assume $\rank(S)\le 23$, since otherwise $S$ coincides with the Leech lattice.

From now on we set $L=\II_{25,1}$, and let $\calD$ be the associated Lobachevsky cone. Denote by $S^{\perp}_\calD$ the subset of $\calD$ consisting of elements orthogonal to $S$. Recall that for a root $\beta$, we write $H_\beta$ to denote the corresponding reflection hyperplane in $(\II_{25,1})_\RR$. 

Then we have $\dim S^{\perp}_\calD\ge 2$ and the intersection $S^{\perp}_\calD \cap H_{\alpha}$ has codimension one in $S^{\perp}_\calD$, namely 
\[
\dim (S^{\perp}_\calD \cap H_{\alpha})=\dim S^{\perp}_\calD -1  \ge 1.
\]

In the original approach by Gaberdiel--Hohenegger--Volpato \cite[Appendix B.2]{gaberdiel2012symmetries}, they found a Conway chamber $\calC_w$ such that $S^{\perp}_\calD \cap \calC_w\ne \emptyset$. Then they utilized Theorem \ref{theorem: conway} to conclude Proposition \ref{proposition: gaberdiel}. The main novelty of this paper is the following claim: \emph{there exists a Conway chamber $\calC_{w}$ for which $S^{\perp}_\calD \cap \calC_w$ is nonempty and $H_\alpha$ is a wall supporting $\calC_w$.} 

The additional requirement on $\calC_w$ that $H_\alpha$ is a wall supporting $\calC_w$ will guarantee that the embedding $S\hookrightarrow \LL$ we are about to construct is primitive.

Next we prove the claim. A key observation is that if a root $\beta\in \II_{25,1}$ satisfies 
\[
S^{\perp}_\calD\cap H_{\alpha} \subset H_\beta,
\]
then $\beta$ lies in $S\oplus \ZZ \alpha$. Since $S$ is positive definite and rootless, it follows that $\beta=\pm \alpha$. 

Therefore, for any root $\beta\ne \pm \alpha$, the intersection $S^{\perp}_\calD \cap H_{\alpha}\cap H_\beta$ is either empty or a  hyperplane in $S^{\perp}_\calD \cap H_{\alpha}$. The union
\[
\bigcup_{\beta\colon \beta^2=2, \beta\ne \pm\alpha} S^{\perp}_\calD \cap H_{\alpha}\cap H_\beta
\]
therefore forms a locally finite hyperplane arrangement in $S^{\perp}_\calD \cap H_{\alpha}$. As a consequence, there exists an element $x\in S^{\perp}_\calD \cap H_{\alpha}$ that does not lie in any reflection hyperplane $H_\beta$ with $\beta\ne \pm \alpha$.

Now take an open neighbourhood $U$ of $x$ in $\calD$, such that $U\cap H_\beta$ is empty for any roots $\beta\ne \pm \alpha$. The intersection $U\cap H_\alpha$ is a nonempty open subset of $H_\alpha$. We write 
\[
U\backslash H_\alpha = U_1\sqcup U_2
\]
where $U_1, U_2$ are open subsets lying on either sides of $H_\alpha$. By shrinking $U$ if necessary, we may assume that both $U_1, U_2$ are connected. Since $S^\perp_\calD$ passes through $x$ and is not contained in $H_\alpha$, it follows that $S^\perp_\calD$ intersects both $U_1$ and $U_2$.

Let $\calC_w$ be the Conway chamber containing $U_1$. Then $H_\alpha$ is wall of $\calC_w$ and $S^{\perp}_\calD \cap \calC_w\ne \emptyset$. The claim is now proved. 

The nonemptyness of $S_\calD^{\perp}\cap \calC_w$ implies that $w\perp S$. The argument is as follows (which is due to \cite{gaberdiel2012symmetries}). The action of $G$ on $S$ extends to an action on $\II_{25,1}$ such that $G$ acts trivially on the orthogonal complement of $S$ in $\II_{25,1}$. In particular, the subset $S^{\perp}_\calD$ is pointwise fixed by $G$. Since $S^{\perp}_\calD \cap \calC_w\ne \emptyset$, it follows that $G$ preserves the Conway chamber $\calC_w$. By the uniqueness of the Weyl vector associated with a Conway chamber (see Theorem \ref{theorem: conway}), we conclude that $G$ fixes $w$. This implies $w\perp S$. Now the composition of
\begin{equation}
\label{equation: embedding}
S\hookrightarrow w^\perp \to w^\perp/ \ZZ w\cong \LL,
\end{equation}
provides a lattice embedding of $S$ into the Leech lattice $\LL$.  

By the claim, we have the fact that $H_\alpha$ is a wall supporting $\calC_w$. We now show that this implies the map $S\hookrightarrow \LL$ from \eqref{equation: embedding} is primitive.

Since $H_\alpha$ is a wall supporting $\calC_w$, we know either $\alpha$ or $-\alpha$ is a Leech root with respect to $w$. Therefore, the vectors $\alpha$ and $w$ span a hyperbolic plane $U=\langle \alpha, w\rangle$, whose orthogonal complement $\langle \alpha, w\rangle^\perp$ in $\II_{25,1}$ is hence positive and unimodular.

We have the following sequence of maps
\[
S\hookrightarrow \langle \alpha, w\rangle^\perp \hookrightarrow w^\perp \to w^\perp/\ZZ w\cong \LL.
\]
The first map $S\hookrightarrow \langle \alpha, w\rangle^\perp$ is primitive. The composition $\langle \alpha, w\rangle^\perp\to \LL$ is an injective morphism between two positive unimodular lattices of the same rank, hence a lattice isomorphism. Thus $S$ admits a primitive embedding into the Leech lattice. This completes the proof for Proposition \ref{proposition: extended version} and Lemma \ref{lemma: main}.

\begin{rmk}
For the case $\rank(S)=24$, conditions $(1)$ and $(3)$ in Proposition \ref{proposition: extended version} are both equivalent to $S\cong \LL$. As a consequence, condition $(2)$ is also equivalent to $S\cong \LL$ in this case.
\end{rmk}

\bigskip

\textbf{Acknowledgement}: I thank Matthias Sch\"utt for pointing out \cite[Corollary 4.19]{marquand2025finite} to me. This work was finished during my visit to University of Chicago in June 2025. I thank Benson Farb for inviting me to University of Chicago, and thank Eduard Looijenga for stimulating discussion on this problem. I thank Chenglong Yu for helpful discussion on hyperplane arrangements and hyperbolic Weyl chambers. I thank Stevell Muller for many useful comments after reading the initial manuscript. The author is supported by NSFC 12301058.

\bibliography{reference}
\end{document}